\pdfoutput=1
\RequirePackage{ifpdf}
\ifpdf 
\documentclass[pdftex]{sigma}
\else
\documentclass{sigma}
\fi

\numberwithin{equation}{section}

\newtheorem{prop}{Proposition}[section]
\newtheorem{lem}[prop]{Lemma}
\newtheorem{cor}[prop]{Corollary}
\newtheorem{thm}[prop]{Theorem}
\newtheorem{assumption}[prop]{Assumption}

{\theoremstyle{definition}
\newtheorem{defn}[prop]{Definition}
\newtheorem{numrmk}[prop]{Remark}
\newtheorem{Example}[prop]{Example}
}

\newcommand{\proj}{\mathbb{P}}

\newcommand{\bP}{\mathbb{P}}

\def\<{\left\langle}
\def\>{\right\rangle}

\def\b1{{\mathbf 1}}

\newcommand{\integer}{\mathbb{Z}}
\newcommand{\rat}{\mathbb{Q}}

\newcommand{\N}{\mathbb{N}}
\newcommand{\Z}{\mathbb{Z}}
\newcommand{\Q}{\mathbb{Q}}
\newcommand{\R}{\mathbb{R}}
\newcommand{\C}{\mathbb{C}}

\newcommand{\bb}{\boldsymbol{b}}

\newcommand{\bx}{\boldsymbol{x}}
\newcommand{\bX}{\mathcal{X}}

\newcommand{\bD}{\mathbb{D}}
\newcommand{\CM}{\mathcal{M}}
\newcommand{\WT}{\widetilde}

\begin{document}

\allowdisplaybreaks

\newcommand{\arXivNumber}{1902.05904}

\renewcommand{\PaperNumber}{055}

\FirstPageHeading

\ShortArticleName{A Note on Disk Counting in Toric Orbifolds}

\ArticleName{A Note on Disk Counting in Toric Orbifolds}

\Author{Kwokwai CHAN~$^{\dag^1}$, Cheol-Hyun CHO~$^{\dag^2}$, Siu-Cheong LAU~$^{\dag^3}$, Naichung Conan LEUNG~$^{\dag^4}$ \newline and Hsian-Hua TSENG~$^{\dag^5}$}

\AuthorNameForHeading{K.~Chan, C.-H.~Cho, S.-C.~Lau, N.-C.~Leung and H.-H.~Tseng}

\Address{$^{\dag^1}$~Department of Mathematics, The Chinese University of Hong Kong, Shatin, Hong Kong}
\EmailDD{\href{mailto:kwchan@math.cuhk.edu.hk}{kwchan@math.cuhk.edu.hk}}

\Address{$^{\dag^2}$~Department of Mathematical Sciences, Research Institute in Mathematics,\\
\hphantom{$^{\dag^2}$}~Seoul National University, Gwanak-Gu, Seoul, South Korea}
\EmailDD{\href{mailto:chocheol@snu.ac.kr}{chocheol@snu.ac.kr}}

\Address{$^{\dag^3}$~Department of Mathematics and Statistics, Boston University, Boston, MA, USA}
\EmailDD{\href{mailto:lau@math.bu.edu}{lau@math.bu.edu}}

\Address{$^{\dag^4}$~The Institute of Mathematical Sciences and Department of Mathematics,\\
\hphantom{$^{\dag^4}$}~The Chinese University of Hong Kong, Shatin, Hong Kong}
\EmailDD{\href{mailto:leung@math.cuhk.edu.hk}{leung@math.cuhk.edu.hk}}

\Address{$^{\dag^5}$~Department of Mathematics, Ohio State University, 100 Math Tower,\\
\hphantom{$^{\dag^5}$}~231 West 18th Ave., Columbus, OH 43210, USA}
\EmailDD{\href{mailto:hhtseng@math.ohio-state.edu}{hhtseng@math.ohio-state.edu}}

\ArticleDates{Received January 24, 2020, in final form June 11, 2020; Published online June 17, 2020}

\Abstract{We compute orbi-disk invariants of compact Gorenstein semi-Fano toric orbifolds by extending the method used for toric Calabi--Yau orbifolds. As a consequence the orbi-disc potential is analytic over complex numbers.}


\Classification{53D37; 14J33}

\section{Introduction}

The mirror map plays a central role in the study of mirror symmetry. It provides a canonical local isomorphism between the K\"ahler moduli and the complex moduli of the mirror near a large complex structure limit. Such an isomorphism is crucial to counting of rational curves using mirror symmetry.

The mirror map is a transformation from the complex coordinates of the Hori--Vafa mirror moduli to the canonical coordinates obtained from period integrals. In \cite{cclt} and \cite{cllt}, we derived an enumerative meaning of the inverse mirror maps for toric Calabi--Yau orbifolds and compact semi-Fano toric manifolds in terms of genus 0 open (orbifold) Gromov--Witten invariants (or (orbi-)disk invariants). Namely, we showed that coefficients of the inverse mirror map are equal to generating functions of virtual counts of stable (orbi-)disks bounded by a regular Lagrangian moment map fiber. In particular it gives a way to effectively compute all such invariants.

It is interesting to compare this with the mirror family constructed by Gross--Siebert \cite{GS07}, which is written in canonical coordinates \cite{RS}. In \cite[Conjecture 0.2]{GS07}, it was conjectured that the wall-crossing functions in their construction are generating functions of open Gromov--Witten invariants. Our results verify this conjecture in the toric setting, namely, we showed that the SYZ mirror family~\cite{SYZ}, constructed using open Gromov--Witten invariants, is written in canonical coordinates.

In this short note we extend our method in \cite{cclt} to derive an explicit formula for the orbi-disk invariants in the case of compact Gorenstein semi-Fano toric orbifolds; see Theorem~\ref{thm:formula} for the explicit formulas. This proves \cite[Conjecture]{cclt-OCRC} for such orbifolds, generalizing \cite[Theorem~1.2]{cllt}:
\begin{thm}[open mirror theorem]For a compact Gorenstein semi-Fano toric orbifold, the orbi-disk potential is equal to the $($extended$)$ Hori--Vafa superpotential via the mirror map.
\end{thm}
See \eqref{eq:W} for the definition of the orbi-disc potential. We remark that the open crepant resolution conjecture \cite[Conjecture~1]{cclt-OCRC} may be studied using this computation and techniques of analytical continuation in \cite[Appendix~A]{cclt}.

\begin{cor}There exists an open neighborhood around the large volume limit where the orbi-disk potential converges.
\end{cor}
This generalizes \cite[Theorem 7.6]{cllt} to the orbifold case.

\section{Preparation}

\subsection{Toric orbifolds}
\subsubsection{Construction}\label{subsec:constr}
Following \cite{BCS}, a {\em stacky fan} is the combinatorial data $(\Sigma, \bb_0,\dots, \bb_{m-1}),$
where $\Sigma$ is a simplicial fan contained in the $\mathbb{R}$-vector space $N_\R := N\otimes_\Z\R$, $N$ is a lattice of rank $n$, and $\{\bb_i \,|\, 0\leq i\leq m-1\}\subset N$ are generators of $1$-dimensional cones of $\Sigma$. $\bb_i$ are called the {\em stacky vectors}.

Choose $\bb_{m}, \dots, \bb_{m'-1}\in N$ so that they are contained in the support of the fan $\Sigma$ and they generate $N$ over $\integer$. An {\em extended stacky fan} in the sense of \cite{Jiang08} is the data
\begin{gather}\label{eqn:ext_stacky_fan}
\big(\Sigma, \{\bb_i\}_{i=0}^{m-1}\cup \{\bb_j\}_{j=m}^{m'-1}\big).
\end{gather}
The vectors $\{\bb_j\}_{j=m}^{m'-1}$ are called {\em extra vectors}.

The {\em fan map} associated to an extended stacky fan (\ref{eqn:ext_stacky_fan}) is defined by
\begin{gather*}\phi\colon \ \widetilde{N}:=\bigoplus_{i=0}^{m'-1}\Z e_i\to N, \qquad \phi(e_i):= \bb_i \quad \textrm{for $i=0,\dots,m'-1$}.\end{gather*}
$\phi$ is surjective and yields an exact sequence of groups called the {\em fan sequence}:
\begin{gather}\label{eqn:fan_seq}
0\longrightarrow \mathbb{L}:=\operatorname{Ker}(\phi)\overset{\psi}{\longrightarrow}\widetilde{N}\overset{\phi}{\longrightarrow} N\longrightarrow 0.
\end{gather}
Clearly $\mathbb{L}\simeq \Z^{m'-n}$. Tensoring (\ref{eqn:fan_seq}) with $\mathbb{C}^\times$ yields the following sequence:
\begin{gather}\label{kexact4}
0 \longrightarrow G:=\mathbb{L}\otimes_\Z \C^\times \longrightarrow \widetilde{N}\otimes_\Z \C^\times \simeq (\C^\times)^{m'} \stackrel{\phi_{\C^\times}}{\longrightarrow} \mathbb{T}:= N\otimes_\Z \C^\times \to 0,
\end{gather}
which is exact. Note that $G$ is an algebraic torus.

By definition, the set of {\em anti-cones} is
\begin{gather*}
\mathcal{A}:=\left\{I\subset \{0, 1, \dots, m'-1\} \,\Big|\, \sum_{i\notin I}\R_{\geq 0} \bb_i \text{ is a cone in } \Sigma\right\}.
\end{gather*}
This terminology is justified because for $I\in \mathcal{A}$, the complement of $I$ in $\{0,1,\dots ,m'-1\}$ indexes generators of a cone in $\Sigma$. For $I\in \mathcal{A}$, the collection $\{Z_i \,|\, i\in I\}$ generates an ideal in $\C[Z_0, \dots, Z_{m'-1}]$, which in turn determines a subvariety $\C^I\subset \C^{m'}$. Set
\begin{gather*}U_\mathcal{A}:=\C^{m'}\setminus \bigcup_{I \notin \mathcal{A}}\C^I.\end{gather*}
The map $G\to (\C^\times)^{m'}$ in \eqref{kexact4} defines a $G$-action on $\C^{m'}$ and hence a $G$-action on $U_\mathcal{A}$. This action is effective and has finite stabilizers, because $N$ is torsion-free (see \cite[Section 2]{Jiang08}). The {\em toric orbifold} associated to $\big(\Sigma, \{\bb_i\}_{i=0}^{m-1}\cup \{\bb_j\}_{j=m}^{m'-1}\big)$ is defined to be the following quotient stack:
\begin{gather*}
\mathcal{X}_\Sigma := [U_\mathcal{A}/G].
\end{gather*}
The standard $(\C^\times)^{m'}$-action on $U_\mathcal{A}$ induces a $\mathbb{T}$-action on $\bX_\Sigma$ via~(\ref{kexact4}).

The coarse moduli space of the toric orbifold $\bX_\Sigma$ is the toric variety $X_\Sigma$ associated to the fan $\Sigma$. In this paper we assume that $X_\Sigma$ is {\em semi-projective}, i.e., $X_\Sigma$ has a $\mathbb{T}$-fixed point and the natural map $X_\Sigma\to \operatorname{Spec}H^0(X_\Sigma, \mathcal{O}_{X_\Sigma})$ is projective, or equivalently, $X_\Sigma$ arises as a GIT quotient of a complex vector space by an abelian group (see \cite[Section~2]{Hausel-Sturmfels}). This assumption is required for the toric mirror theorem of~\cite{CCIT_toricDM} to hold. More detailed discussions on semi-projective toric varieties can be found in \cite[Section~7.2]{CLS_toricbook}.

\subsubsection{Twisted sectors}\label{subsec:tw_sec}
Consider a $d$-dimensional cone $\sigma\in \Sigma$ generated by $\bb_\sigma=(\bb_{i_1}, \dots, \bb_{i_d})$. Define
\begin{gather*}{\rm Box}_{\bb_\sigma} :=\left\{\nu \in N \,\Big|\, \nu=\sum_{k=1}^d t_k \bb_{i_k},\, t_k \in [0,1)\cap\rat\right\}.\end{gather*}
$\{ \bb_{i_1}, \dots, \bb_{i_d}\}$ generates a submodule $N_{\bb_\sigma}\subset N$. One can check that there is a bijection between ${\rm Box}_{\bb_\sigma}$ and the finite group $G_{\bb_{\sigma}} := (N \cap \mathrm{Span}_\R \bb_\sigma)/N_{\bb_{\sigma}}$. Furthermore, if $\tau$ is a subcone of $\sigma$, then ${\rm Box}_{\bb_\tau}\subset {\rm Box}_{\bb_\sigma}$. Define
\begin{gather*}
{\rm Box}_{\bb_\sigma}^{\circ}\colon \
 {\rm Box}_{\bb_\sigma} \setminus \bigcup_{\tau \precneqq \sigma} {\rm Box}_{\bb_\tau},\qquad {\rm Box}(\Sigma) := \bigcup_{\sigma \in \Sigma^{(n)}} {\rm Box}_{\bb_\sigma} = \bigsqcup_{\sigma \in \Sigma} {\rm Box}_{\bb_\sigma}^{\circ}, \\ {\rm Box}'(\Sigma)={\rm Box}(\Sigma)\setminus\{0\},
\end{gather*}
where $\Sigma^{(n)}$ is the set of $n$-dimensional cones in $\Sigma$.

 Following the description of the inertia orbifold of $\bX_\Sigma$ in \cite{BCS}, for $\nu\in{\rm Box}(\Sigma)$, we denote by $\bX_\nu$ the corresponding component of the inertia orbifold of $\bX:=\bX_\Sigma$. Note that $\bX_0=\bX_\Sigma$ as orbifolds. Elements $\nu\in {\rm Box}'(\Sigma)$ correspond to {\em twisted sectors} of $\bX$, namely non-trivial connected components of the inertia orbifold of $\bX$.

 Following \cite{CR04}, the direct sum of singular cohomology groups of components of the inertia orbifold of $\bX$, subject to a degree shift, is called the {\em Chen--Ruan orbifold cohomology} $H^*_\mathrm{CR}(\mathcal{X};\rat)$ of $\bX$. More precisely,
\begin{gather*}H^d_\mathrm{CR}(\mathcal{X};\rat)=\bigoplus_{\nu\in{\rm Box}}H^{d-2\operatorname{age}(\nu)}(\mathcal{X}_\nu;\rat),\end{gather*}
where $\operatorname{age}(\nu)$ is called the {\em degree shifting number}\footnote{Following Miles Reid, it is now more commonly called {\em age}.} in \cite{CR04} of the twisted sector $\mathcal{X}_\nu$. In case of toric orbifolds, $\operatorname{age}$ has a combinatorial description \cite{BCS}: if $\nu=\sum\limits_{k=1}^d t_k \bb_{i_k} \in {\rm Box}(\Sigma)$ where $\{\bb_{i_1},\dots,\bb_{i_d}\}$ generates a cone in $\Sigma$, then \begin{gather*}\operatorname{age}(\nu) = \sum_{k=1}^d t_k \in \rat_{\geq0}.\end{gather*}

Using $\mathbb{T}$-actions on twisted sectors induced from that on $\bX$, we can define {\em $\mathbb{T}$-equivariant Chen--Ruan orbifold cohomology} $H^*_{\mathrm{CR}, \mathbb{T}}(\mathcal{X};\rat)$ by replacing singular cohomology with $\mathbb{T}$-equivariant cohomology $H^*_\mathbb{T}(-)$. Namely
\begin{gather*}H^d_{\mathrm{CR}, \mathbb{T}}(\mathcal{X};\rat)=\bigoplus_{\nu\in{\rm Box}}H^{d-2\operatorname{age}(\nu)}_{\mathbb{T}}(\mathcal{X}_\nu;\rat).\end{gather*}
By general properties of equivariant cohomology, $H^*_{\mathrm{CR}, \mathbb{T}}(\mathcal{X};\rat)$ is a module over $H^*_\mathbb{T}(\text{pt}, \mathbb{Q})$ and admits a map $H^*_{\mathrm{CR}, \mathbb{T}}(\mathcal{X};\rat)\to H^*_{\mathrm{CR}}(\mathcal{X};\rat)$ called {\em non-equivariant limit}.

\subsubsection{Toric divisors, K\"ahler cones, and Mori cones}\label{sec:div_cone_etc}

We continue using the notations in Sections \ref{subsec:constr} and~\ref{subsec:tw_sec}. Applying $\operatorname{Hom}_\Z(-,\Z)$ to the fan sequence (\ref{eqn:fan_seq}), we obtain the following exact sequence:\footnote{The map $\psi^\vee: \widetilde{M}\to \mathbb{L}^\vee$ is surjective since $N$ is torsion-free.}
\begin{gather*}
0\longrightarrow M:= N^\vee = \operatorname{Hom}(N, \Z) \overset{\phi^\vee}{\longrightarrow} \widetilde{M}:= \widetilde{N}^\vee = \operatorname{Hom}\big(\widetilde{N},\Z\big)\overset{\psi^\vee}{\longrightarrow} \mathbb{L}^\vee= \operatorname{Hom}(\mathbb{L},\Z) \longrightarrow 0,
\end{gather*}
which is called the {\em divisor sequence}. Line bundles on $\bX=[\mathcal{U}_\mathcal{A}/G]$ correspond to $G$-equivariant line bundles on $\mathcal{U}_\mathcal{A}$. In view of (\ref{kexact4}), $\mathbb{T}$-equivariant line bundles on $\bX$ correspond to $(\mathbb{C}^\times)^{m'}$-equivariant line bundles on $\mathcal{U}_\mathcal{A}$. Because the codimension of $\cup_{I\notin \mathcal{A}} \mathbb{C}^I\subset \mathbb{C}^{m'}$ is at least $2$, the Picard groups satisfy: \begin{gather*}
\operatorname{Pic}(\bX)\simeq \operatorname{Hom}(G, \mathbb{C}^\times)\simeq \mathbb{L}^\vee, \qquad \operatorname{Pic}_\mathbb{T}(\bX)\simeq \operatorname{Hom}\big((\mathbb{C}^\times)^{m'}, \mathbb{C}^\times\big)\simeq \widetilde{N}^\vee=\widetilde{M}.\end{gather*}
The natural map $\operatorname{Pic}_\mathbb{T}(\bX)\to \operatorname{Pic}(\bX)$ is identified with the map $\psi^\vee\colon \widetilde{M}\to \mathbb{L}^\vee$ appearing in the divisor sequence.

The elements $\big\{e_i^\vee \,|\, i=0, 1, \dots, m'-1 \big\}\subset \widetilde{M}\simeq \operatorname{Pic}_\mathbb{T}(\bX)$ dual to $\{e_i \,|\, i=0, 1, \dots, m'-1\}\subset \widetilde{N}$ correspond to $\mathbb{T}$-equivariant line bundle on $\bX$ which we denote by $D_i^\mathbb{T}$, $i=0, 1, \dots, m'-1$. The collection \begin{gather*}\big\{D_i:=\psi^\vee\big(e_i^\vee\big) \,|\, 0\leq i\leq m-1\big\}\subset \mathbb{L}^\vee \simeq \operatorname{Pic}(\bX)\end{gather*} consists of toric prime divisors corresponding to the generators $\{\bb_i \,|\, 0\leq i\leq m-1\}$ of $1$-di\-men\-sional cones in $\Sigma$. Elements $D^\mathbb{T}_i$, $0\leq i\leq m-1$ are $\mathbb{T}$-equivariant lifts of these divisors. There are natural maps
\begin{gather*}
\widetilde{M}\otimes \mathbb{Q}\overset{\psi^\vee\otimes \mathbb{Q}}{\rightarrow} \mathbb{L}^\vee\otimes \mathbb{Q},\\
\big(\widetilde{M}\otimes\mathbb{Q}\big) \Big/ \left(\sum_{j=m}^{m'-1}\mathbb{Q}D^\mathbb{T}_j\right)\simeq H_\mathbb{T}^2(\bX, \mathbb{Q})\to H^2(\bX, \mathbb{Q}) \simeq \big(\mathbb{L}^\vee\otimes\mathbb{Q}\big) \Big/ \left(\sum_{j=m}^{m'-1}\mathbb{Q}D_j\right).
\end{gather*}
Together with the natural quotient maps, they fit into a commutative diagram.

As explained in \cite[Section 3.1.2]{iritani09}, there is a canonical splitting of the quotient map $\mathbb{L}^\vee\otimes \mathbb{Q}\to H^2(\bX; \mathbb{Q})$. For $m\leq j \leq m'-1$, let $I_j\in \mathcal{A}$ be the anticone of the cone containing $\bb_j$. This allows us to write
$\bb_j=\sum\limits_{i\notin I_j} c_{ji}\bb_i$ for $c_{ji}\in\rat_{\geq 0}$.

Tensoring the fan sequence \eqref{eqn:fan_seq} with $\mathbb{Q}$, we may find a unique $D_j^\vee\in \mathbb{L}\otimes \mathbb{Q}$ such that values of the natural pairing $\langle -,-\rangle$ between $\mathbb{L}^\vee$ and $\mathbb{L}$ satisfy
\begin{gather}\label{eqn:dual_of_D_j}
\langle D_i, D_j^\vee\rangle
=
\begin{cases}
1 & \textrm{if } i=j,\\
-c_{ji} & \textrm{if } i\notin I_j,\\
0 & \textrm{if } i\in I_j\setminus \{j\}.
\end{cases}
\end{gather}
Using $D_j^\vee$ we get a decomposition
\begin{gather}\label{eqn:splitting_L_dual}
\mathbb{L}^\vee\otimes \mathbb{Q}
=\operatorname{Ker}\big(\big(D_{m}^\vee, \dots, D_{m'-1}^\vee \big)\colon \mathbb{L}^\vee\otimes \mathbb{Q}\to \mathbb{Q}^{m'-m}\big)\oplus \bigoplus_{j=m}^{m'-1}\mathbb{Q}D_j.
\end{gather}
We can view $H^2(\bX; \mathbb{Q})$ as a subspace of $\mathbb{L}^\vee\otimes\mathbb{Q}$ because $\operatorname{Ker}\big(\big(D_{m}^\vee, \dots, D_{m'-1}^\vee \big)\big)$ can be identified with $H^2(\bX; \mathbb{Q})$ via the map $\mathbb{L}^\vee\otimes\mathbb{Q}\to H^2(\bX; \mathbb{Q})$.

Define {\em extended K\"ahler cone} of $\bX$ to be
\begin{gather*}\widetilde{C}_\bX := \bigcap_{I\in \mathcal{A}}\left(\sum_{i\in I}\mathbb{R}_{>0}D_i \right)\subset \mathbb{L}^\vee\otimes\mathbb{R}.\end{gather*}
The K\"ahler cone $C_\bX$ is the image of $\widetilde{C}_\bX$ under $\mathbb{L}^\vee\otimes \mathbb{R}\to H^2(\bX; \R)$. The splitting \eqref{eqn:splitting_L_dual} of $\mathbb{L}^\vee\otimes\mathbb{Q}$ yields a splitting $\widetilde{C}_\bX = C_\bX+\sum\limits_{j=m}^{m'-1} \mathbb{R}_{>0}D_j.$

By (\ref{eqn:fan_seq}), $\mathbb{L}^\vee$ has rank equal to $r:=m'-n$. The rank of $H_2(\bX; \mathbb{Z})$ is $r':=r-(m'-m)=m-n.$ We choose an integral basis
\begin{gather*}\{p_1, \dots, p_r\}\subset \mathbb{L}^\vee,\end{gather*}
such that $p_a$ is in the closure of $\widetilde{C}_\bX$ for all $a$ and $p_{r'+1}, \dots, p_r \in \sum\limits_{i=m}^{m'-1}\mathbb{R}_{\geq 0}D_i$. We get a nef basis $\{\bar{p}_1,\dots,\bar{p}_{r'}\}$ for $H^2(\bX;\Q)$ as images of $\{p_1,\dots,p_{r'}\}$ under the quotient map $\mathbb{L}^\vee\otimes \Q\to H^2(\bX; \Q)$. For $r'+1\leq a\leq r$, the images satisfies $\bar{p}_a = 0$.

We choose equivariant lifts of $p_a$'s, namely $\big\{p_1^\mathbb{T}, \dots, p_r^\mathbb{T}\big\}\subset \widetilde{M}\otimes \mathbb{Q}$ such that $\psi^\vee\big(p_a^\mathbb{T}\big)=p_a$ for all $a$. We also require that for $a=r'+1,\dots ,r$ the images $\bar{p}_a^{\mathbb{T}}$ of $p_a^\mathbb{T}$ under the natural map $\widetilde{M}\otimes \mathbb{Q}\to H_\mathbb{T}^2(\bX, \mathbb{Q})$ satisfies $\bar{p}_a^{\mathbb{T}}=0$.

The coefficients $Q_{ia}\in\mathbb{Z}$ in the equations $D_i = \sum\limits_{a=1}^r Q_{ia} p_a$ assemble to a matrix $(Q_{ia})$. The images\footnote{$\bar{D}_i$ is the class of the toric prime divisor $D_i$.} $\bar{D}_i$ of $D_i$ under the map $\mathbb{L}^\vee\otimes \Q\to H^2(\bX; \Q)$ can be expressed as
\begin{gather*}\bar{D}_i = \sum_{a=1}^{r'} Q_{ia}\bar{p}_a, \qquad i=0,\dots,m-1.\end{gather*}
Their equivariant lifts $\bar{D}^\mathbb{T}_i$ can be expressed as
\begin{gather*}\bar{D}^\mathbb{T}_i = \sum_{a=1}^{r'} Q_{ia}\bar{p}^\mathbb{T}_a+\lambda_i, \qquad \text{where} \quad \lambda_i\in H^2(B\mathbb{T}; \mathbb{Q}).\end{gather*}
For $i=m,\dots, m'-1$, we have $\bar{D}_i=0$ in $H^2(\bX; \R)$ and $\bar{D}^\mathbb{T}_i=0$.

Localization gives the following description of $H_{\mathrm{CR}, \mathbb{T}}^{\leq 2}$:
\begin{gather*}H^0_{\mathrm{CR}, \mathbb{T}}(\bX, K_\mathbb{T})=K_\mathbb{T}{\bf 1}, \qquad H^2_{\mathrm{CR}, \mathbb{T}}(\bX, K_\mathbb{T})=\bigoplus_{a=1}^{r'} K_\mathbb{T} \bar{p}_a^{\mathbb{T}}\oplus \bigoplus_{\nu\in \text{Box}, \text{age}(\nu)=1} K_\mathbb{T} {\bf 1}_\nu.\end{gather*}
Here $K_\mathbb{T}$ is the field of fractions of $H_\mathbb{T}^*(\text{pt}, \mathbb{Q})$, ${\bf 1}\in H^0(\bX, \rat)$ and ${\bf 1}_\nu\in H^0(\bX_\nu, \rat)$ are
fundamental classes.

Let
\begin{gather*}\{\gamma_1, \dots, \gamma_r\} \subset \mathbb{L}, \qquad \gamma_a = \sum_{i=0}^{m'-1} Q_{ia}e_i \in \widetilde{N},\end{gather*}
be the basis dual to $\{p_1, \dots, p_r\}\subset \mathbb{L}^\vee$. $H_2^{\rm eff}(\bX;\Q)$ admits a basis $\{\gamma_1,\dots,\gamma_{r'}\}$, and we have $Q_{ia} = 0$ when $m\leq i\leq m'-1$ and $1\leq a\leq r'$.

Set
\begin{gather*}
\mathbb{K} :=\{d\in\mathbb{L}\otimes\rat \,|\, \{j\in\{0, 1,\dots,m'-1\}\,|\,\langle D_j,d\rangle\in\Z\}\in\mathcal{A}\},\\
\mathbb{K}_{\rm eff} :=\{d\in\mathbb{L}\otimes\rat \,|\, \{j\in\{0,1,\dots,m'-1\}\,|\,\langle D_j,d\rangle\in\Z_{\geq0}\}\in\mathcal{A}\}.
\end{gather*}
Elements of $\mathbb{K}_{\rm eff}$ should be interpreted as effective curve classes. Elements of $\mathbb{K}_{\rm eff} \cap H_2(\bX;\R)$ should be viewed as classes of stable maps $\proj(1,m) \to \bX$ for some $m \in \Z_{\geq0}$. See, e.g., \cite[Section~3.1]{iritani09} for more details.

\begin{defn} \label{defn:sF}
A toric orbifold $\bX$ is called {\em semi-Fano} if $c_1(\bX) \cdot \alpha > 0$ for every effective curve class $\alpha$, in other words, $-K_\bX$ is nef.
\end{defn}

For $d\in\mathbb{K}$, put\footnote{For a real number $\lambda\in\R$, let $\lceil \lambda \rceil$, $\lfloor \lambda \rfloor$ and $\{\lambda\}$ denote the ceiling, floor and fractional part of $\lambda$ respectively.}
\begin{gather*}
\nu(d):=\sum_{i=0}^{m'-1}\lceil\langle D_i,d\rangle\rceil\bb_i\in N,
\end{gather*}
and let $I_d := \{j\in\{0, 1,\dots,m'-1\}\,|\,\langle D_j,d\rangle\in\Z\} \in \mathcal{A}$. Then $\nu(d)\in{\rm Box}$ because
\begin{gather*}\nu(d) = \sum_{i=0}^{m'-1}(\{-\langle D_i,d\rangle\}+\langle D_i,d\rangle)\bb_i
= \sum_{i=0}^{m'-1} \{-\langle D_i,d\rangle\}\bb_i
= \sum_{i\notin I_d} \{-\langle D_i,d\rangle\}\bb_i.\end{gather*}

\subsection{Genus 0 open orbifold GW invariants according to \cite{CP}}\label{sec:review_disk_inv}

Let $(\bX,\omega)$ be a toric K\"ahler orbifold of complex dimension $n$, equipped with the standard toric complex structure $J_0$ and a toric K\"ahler structure $\omega$. Denote by $(\Sigma,\bb)$ the stacky fan that defines $\bX$, where $\bb=(\bb_0,\dots,\bb_{m-1})$ and $\bb_i=c_iv_i$.

Let $L \subset \bX$ be a Lagrangian torus fiber of the moment map $\mu_0\colon \bX \to M_\R := M\otimes_\Z \R$, and let $\beta \in \pi_2(\bX,L) = H_2(\bX,L;\integer)$ be a relative homotopy class.

\subsubsection{Holomorphic orbi-disks and their moduli spaces}\label{sec:orbidisk_and_moduli}
A {\em holomorphic orbi-disk} in $\bX$ with boundary in $L$ is a continuous map
\begin{gather*}w\colon \ (\bD,\partial\bD) \to (\bX,L)\end{gather*}
satisfying the following conditions:
\begin{enumerate}\itemsep=0pt
\item
$\big(\bD, z_1^+,\dots,z_l^+\big)$ is an {\em orbi-disk} with interior marked points $z_1^+,\dots,z_l^+$. More precisely $\bD$ is analytically the disk $D^2\subset\C$ so that for $j=1,\dots,l$, the orbifold structure at $z_j^+$ is given by a disk neighborhood of $z_j^+$ uniformized by the branched covering map $\operatorname{br}\colon z \to z^{m_j}$ for some $m_j\in \integer_{>0}$. (If $m_j=1$, $z_j^+$ is not an orbifold point.)
\item
For any $z_0 \in \bD$, there is a disk neighborhood of $z_0$ with a branched covering map $\operatorname{br} \colon z \allowbreak \to z^m$, and there is a local chart $\big(V_{w(z_0)},G_{w(z_0)},\pi_{w(z_0)}\big)$ of $\mathcal{X}$ at $w(z_0)$ and a local holomorphic lifting $\WT{w}_{z_0}$ of $w$ satisfying
$w \circ \operatorname{br} = \pi_{w(z_0)} \circ \WT{w}_{z_0}.$

\item
The map $w$ is {\em good} (in the sense of Chen--Ruan \cite{CR01}) and {\em representable}. In particular, for each $z_j^+$, the associated group homomorphism
\begin{gather*}
h_p\colon \ \Z_{m_j}\to G_{w(z_j^+)}
\end{gather*}
between local groups which makes $\WT{w}_{z_j^+}$ equivariant, is {\em injective}.
\end{enumerate}

The {\em type} of a map $w$ as above is defined to be $\bx := (\bX_{\nu_1},\dots, \bX_{\nu_l})$. Here $\nu_j\in{\rm Box}(\Sigma)$ is the image of the generator $1\in \Z_{m_j}$ under $h_j$.

There are two notions of Maslov index for an orbi-disk. The {\em desingularized Maslov index} $\mu^{\rm de}$ is defined by desingularizing the interior singularities of the pull-back bundle $w^*T\bX$. Namely, the bundle $w^*T\bX$ over an orbi-disk $\big(\bD, z_1^+,\dots,z_l^+\big)$ cannot be trivialized due to the orbifold structure, but we can obtain another bundle $|w^*T\bX|$ by modifying the bundle near orbifold points (see Chen--Ruan \cite{CR01} for more details). This is called a desingularization of $w^*T\bX$ and it is a smooth bundle over the orbi-disk, hence is a trivial bundle. We can compute the Maslov index of the boundary Lagrangian loop relative to this trivialization, and it is called the desingularized Maslov index. See \cite[Section~3]{CP} for more details and \cite[Section~5]{CP} for an explicit formula in the toric case.

The {\em Chern--Weil $($CW$)$ Maslov index} $\mu_{\rm CW}$ is defined as the integral of the curvature of a~unitary connection on $w^*T\bX$ which preserves the Lagrangian boundary condition, see~\cite{CS} (and also \cite[Section~3.3]{CP} for a relation with $\mu^{\rm de}$). The following lemma, which appeared as \cite[Lemma~3.1]{cclt}, computes the CW Maslov indices of disks. This is an orbifold version of the formula in \cite[Lemma~3.1]{auroux07}.
\begin{lem}\label{lem:maslov_ind_comp}
Let $(\bX, \omega, J)$ be a K\"ahler orbifold of complex dimension $n$. Let $\Omega$ be a non-zero meromorphic $n$-form on $\bX$ which has at worst simple poles. Let $D\subset \bX$ be the pole divisor of $\Omega$. Suppose also that the generic points of $D$ are smooth. Then for a special Lagrangian submanifold $L \subset \bX \setminus D$, the CW Maslov index of a class $\beta \in \pi_2(\bX,L)$ is given by
\begin{gather*}
\mu_{\rm CW}(\beta) = 2\beta \cdot D.
\end{gather*}
Here, $\beta \cdot D$ is defined by writing $\beta$ as a fractional linear combination of homotopy classes of smooth disks.
\end{lem}

The classification of orbi-disks in a symplectic toric orbifold has been worked out in \cite[Theorem~6.2]{CP}. It is similar to the classification of holomorphic discs in toric manifolds~\cite{cho06}. In the classification, the {\em basic disks} corresponding to the stacky vectors (and twisted sectors) play a basic role.
\begin{thm}[{\cite[Corollaries 6.3 and 6.4]{CP}}]
Let $\bX$ be a toric K\"ahler orbifold and let $L$ be a~fiber of the toric moment map.
\begin{enumerate}\itemsep=0pt
\item[$1.$] The {\em smooth} holomorphic disks of Maslov index $2$ $($modulo $\mathbb{T}^n$-action and reparametrizations of the domain$)$ are in bijective correspondence with the stacky vectors $\{\bb_0,\dots,\bb_{m-1}\}$. Denote the homotopy classes of these disks by $\beta_{0},\dots,\beta_{m-1}$.
\item[$2.$] The holomorphic orbi-disks with one interior orbifold marked point and desingularized Maslov index~$0$ $($modulo $\mathbb{T}^n$-action and reparametrizations of the domain$)$ are in bijective correspondence with the twisted sectors $\nu \in {\rm Box}'(\Sigma)$ of the toric orbifold $\bX$. Denote the homotopy classes of these orbi-disks by $\beta_\nu$.
\end{enumerate}
\end{thm}

\begin{lem}[{\cite[Lemma~9.1]{CP}}]\label{lem:basic_disk_classes}
For $\bX$ and $L$ as above, the relative homotopy group $\pi_2(\bX,L)$ is generated by the classes $\beta_i$ for $i=0,\dots,m-1$ together with $\beta_\nu$ for $\nu\in{\rm Box}'(\Sigma)$.
\end{lem}

As in \cite{CP}, these generators of $\pi_2(\bX,L)$ are called {\em basic disk classes}. They are the analogue of Maslov index $2$ disk classes in toric manifolds.

Let
\[ \CM^{\rm op, main}_{k+1,l}(\bX, L,\beta,\bx)\]
be the moduli space of good representable stable maps from bordered orbifold Riemann surfaces of genus zero with $k+1$ boundary marked points $z_0,z_1,\dots,z_k$ and $l$ interior (orbifold) marked points $z_1^+,\dots,z_l^+$ in the homotopy class $\beta$ of type $\bx = (\mathcal{X}_{\nu_1},\dots, \mathcal{X}_{\nu_l})$. The superscript ``$main$" is meant to indicate the connected component on which the boundary marked points respect the cyclic order of $S^1=\partial D^2$.
According to \cite[Lemma~2.5]{CP}, $\CM^{\rm op, main}_{k+1,l}(\bX, L,\beta,\bx)$ has real virtual dimension
\begin{gather*}
n + \mu_{\rm CW}(\beta) + k + 1 + 2l -3 - 2\sum_{j=1}^l\operatorname{age}(\nu_j).
\end{gather*}

By \cite[Proposition~9.4]{CP}, if $\CM^{\rm op, main}_{1,1}(\bX, L,\beta)$ is non-empty and if $\partial\beta$ is not in the sublattice generated by $\bb_0, \dots, \bb_{m-1}$, then there exist $\nu\in {\rm Box}'(\Sigma)$, $k_0,\dots ,k_{m-1}\in\N$ and $\alpha\in H_2^{\rm eff}(\bX)$ such that
$\beta = \beta_\nu + \sum\limits_{i=0}^{m-1} k_i \beta_i + \alpha,$
where $\alpha$ is realized by a union of holomorphic (orbi-)spheres. The CW Maslov index of $\beta$ written in this way is given by
$\mu_{\rm CW}(\beta) = 2\operatorname{age}(\nu) + 2 \sum\limits_{i=0}^{m-1} k_i + 2 c_1(\bX)\cdot\alpha$.

\subsubsection{Orbi-disk invariants}\label{subsec:disk_inv}
Pick twisted sectors $\bX_{\nu_1}, \dots, \bX_{\nu_l}$ of the toric orbifold $\bX$. Consider the moduli space \begin{gather*}\CM^{\rm op, main}_{1,l}(\bX, L,\beta,\bx)\end{gather*} of good representable stable maps from bordered orbifold Riemann surfaces of genus zero with one boundary marked point and $l$ interior orbifold marked points of type $\bx = (\bX_{\nu_1},\dots, \bX_{\nu_l})$ representing the class $\beta\in\pi_2(\bX, L)$. According to \cite{CP}, $\CM^{\rm op, main}_{1,l}(\bX, L,\beta,\bx)$ can be equipped with a virtual fundamental chain, which has an expected dimension $n$ if the following equality holds:
\begin{gather}\label{critmaslov}
\mu_{\rm CW}(\beta) = 2 + \sum_{j=1}^l (2\cdot\operatorname{age}(\nu_j)-2).
\end{gather}

Throughout the paper, we make the following assumptions.

\begin{assumption}We assume that the toric orbifold $\bX$ is semi-Fano $($see Definition~{\rm \ref{defn:sF})} and {\em Gorenstein}.\footnote{This means that $K_\bX$ is Cartier.}	Moreover, we assume that the type $\bx$ consists of twisted sectors with $\operatorname{age}\leq 1$.\footnote{This assumption does not impose any restriction in the construction of the SYZ mirror over $H^{\leq 2}_{\mathrm{CR}}(\bX)$. We do not discuss mirror construction in this paper.}
\end{assumption}

Then the age of every twisted sector of $\bX$ is a non-negative integer. Since a basic orbi-disk class $\beta_\nu$ has Maslov index $2 \operatorname{age}(\nu)$, we see that every non-constant stable disk class has at least Maslov index $2$.

Moreover, the virtual fundamental chain $\big[\CM^{\rm op, main}_{1,l}(\bX, L,\beta,\bx)\big]^{\rm vir}$ has expected dimension $n$ when $\mu_{\rm CW}(\beta) = 2$, and in fact we get a virtual fundamental {\em cycle} because $\beta$ attains the minimal Maslov index, thus preventing disk bubbling to occur. Therefore the following definition of {\em genus~$0$ open orbifold GW invariants} (also known as {\em orbi-disk invariants}) is independent of the choice of perturbations of the Kuranishi structures:\footnote{In the general case one may restrict to torus-equivariant perturbations, as did in \cite{FOOO1, FOOO2, FOOO10b}.}

\begin{defn}[orbi-disk invariants]\label{defn:orbi-disk_inv}
Let $\beta\in\pi_2(\bX,L)$ be a relative homotopy class with Maslov index given by \eqref{critmaslov}. Suppose that the moduli space $\CM^{\rm op, main}_{1,l}(\bX, L,\beta,\bx)$ has a virtual fundamental cycle of dimension $n$. Then we define
\begin{gather*}n_{1,l,\beta}^\bX([{\rm pt}]_{L};\mathbf{1}_{\nu_1},\dots,\mathbf{1}_{\nu_l}):=
\operatorname{ev}_{0*}\big(\big[\CM^{\rm op, main}_{1,l}(\bX, L,\beta,\bx)\big]^{\rm vir}\big)\in H_n(L;\rat)\cong\rat,\end{gather*}
where $\operatorname{ev}_0\colon \CM^{\rm op, main}_{1,l}(\bX, L,\beta,\bx)\to L$ is the evaluation map at the boundary marked point, $[{\rm pt}]_{L} \in H^n(L;\rat)$ is the point class of $L$, and $\mathbf{1}_{\nu_j} \in H^0(\bX_{\nu_j};\rat)\subset H^{2\operatorname{age}(\nu_j)}_{\mathrm{CR}}(\bX;\rat)$ is the fundamental class of the twisted sector $\bX_{\nu_j}$.
\end{defn}

\begin{numrmk}The Kuranishi structures in this paper are the same as those defined in \cite{FOOO1, FOOO2}, incorporating the works \cite{CR01,CR04} for the interior orbifold marked points. This has been explained in \cite[Section~10]{CP}. We also refer the readers to \cite[Appendix]{FOOO_book} and \cite{FOOO12} for the detailed construction, and to \cite{McDuff-Wehrheim12} (and its forthcoming sequels) for a different approach.

The moduli spaces considered here are in fact much simpler than those in \cite{FOOO1, FOOO2} (and \cite{FOOO_book}) because we only need to consider stable disks with just one disk component which is minimal, and hence disk bubbling does not occur. In particular, we do not have codimension-one boundary components, and hence the above definition is independent of choices of Kuranishi perturbations.
\end{numrmk}

For a basic (orbi-)disk with at most one interior orbifold marked point, the corresponding moduli space $\CM^{\rm op, main}_{1,0}(\bX, L,\beta_i)$ (or $\CM^{\rm op, main}_{1,1}(\bX, L,\beta_\nu,\nu)$ when $\beta_\nu$ is a basic orbi-disk class) is regular and can be identified with $L$. Thus the associated invariants are evaluated as follows~\cite{CP}:
\begin{enumerate}\itemsep=0pt
\item For $\nu \in {\rm Box}'$, we have $n_{1,1,\beta_\nu}^\bX([{\rm pt}]_{L};\mathbf{1}_\nu) = 1.$
\item For $i\in \{0,\dots,m-1\}$, we have $n_{1,0,\beta_i}^\bX([{\rm pt}]_{L}) = 1.$
\end{enumerate}
When there are more interior orbifold marked points or when the disk class is not basic, the corresponding moduli space is in general non-regular and virtual theory is involved in the definition, making the invariant much more difficult to compute.

\section{Geometric constructions}
Let $\beta\in \pi_2(\bX, L)$ be a disk class with $\mu_{\rm CW}(\beta)=2$. By the discussion in Section \ref{sec:review_disk_inv}, we can write \begin{gather*}\beta=\beta_{\mathbf d}+\alpha\end{gather*}
with $\alpha\in H_2(\bX, \mathbb{Z})$, $c_1(\bX)\cdot \alpha=0$ and either $\beta_{\mathbf d}\in \{\beta_0,\dots ,\beta_{m-1}\}$ or $\beta_{\mathbf d}\in \text{Box}'(\bX)$ is of age $1$. Denote by $\bb_d\in N$ the element corresponding to $\beta_{\mathbf d}$.

Recall that the fan polytope $\mathcal{P}\subset N_\mathbb{R}$ is the convex hull of the vectors $\bb_0,\dots , \bb_{m-1}$. Note that $\bb_d\in \mathcal{P}$. Denote by $F(\bb_d)$ the minimal face of the fan polytope $\mathcal{P}$ that contains the vector~$\bb_d$. Let~$F$ be a facet of $\mathcal{P}$ that contains $F(\bb_d)$. Let $\Sigma_{\beta_{\mathbf d}}\subset \Sigma$ be the minimal convex subfan containing all $\{\bb_0,\dots , \bb_{m-1}\}\cap F$. The vectors \begin{gather}
\{\bb_0,\dots ,\bb_{m'-1}\}\cap \sum_{\bb_j\in \{\bb_0,\dots , \bb_{m-1}\}\cap F} \mathbb{Q}_{\geq 0}\bb_j
\label{eq:toricCY}
\end{gather} determine a fan map $\Z^p \to N$ (where $p$ is the number of vectors above). Let \begin{gather*}\bX_{\beta_{\mathbf d}}\subset \bX\end{gather*} be the associated toric suborbifold (of the same dimension $n$).

\begin{lem}
$\bX_{\beta_{\mathbf d}}$ is a toric Calabi--Yau orbifold.
\end{lem}

\begin{proof}
	All the generators in \eqref{eq:toricCY} lie in the hyperplane containing $F$. Since $\bX$ is Gorenstein, this hyperplane has a defining equation $\nu=0$ for some primitive vector $\nu \in M$. Hence $\bX_{\beta_{\mathbf d}}$ is toric Calabi--Yau.
\end{proof}

\begin{Example}Consider $\bP^2/\Z_3$, whose fan is shown in the left of Fig.~\ref{fig:P2mod3}. If $\beta_{\mathbf d}$ corresponds to the vector $(1,0)$ (which is marked as `113' in the figure), then $\Sigma_{\beta_{\mathbf d}}$ is the cone spanned by $v_2 = (2,-1)$ and $v_3 = (-1,2)$. If $\beta_{\mathbf d}$ corresponds to the vector $v_3$, then $\Sigma_{\beta_{\mathbf d}}$ can be taken to be the cone spanned by $v_2$, $v_3$, or the cone spanned by $v_1$, $v_3$. In both cases, the corresponding toric Calabi--Yau orbifold is $\C^2/\Z_3$.
\end{Example}

Note that $\bX_{\beta_{\mathbf d}}$ depends on the choice of the face $F$, not just $\beta_{\mathbf d}$. We use $\bX_{\beta_{\mathbf d}}$ to compute open Gromov--Witten invariants of $\bX$ in class $\beta=\beta_{\mathbf d}+\alpha$.

In what follows we show that $\bX_{\beta_{\mathbf d}}\subset \bX$ contains all stable orbi-disks of $\bX$ of class $\beta$. First, we have the following analogue of \cite[Proposition 5.6]{cllt}.
\begin{lem}\label{lem:in_div}
Let $f\colon \mathcal{D}\cup \mathcal{C}\to \bX$ be a stable orbi-disk map in the class $\beta=\beta_{\mathbf d}+\alpha$, where $\mathcal{D}$ is a~$($possibly orbifold$)$ disk and $\mathcal{C}$ is a~$($possibly orbifold$)$ rational curve such that $f_*[\mathcal{D}]=\beta_{\mathbf d}$ and $f_*[\mathcal{C}]=\alpha$ with $c_1(\alpha)=0$. Then we have \begin{gather*}f(\mathcal{C})\subset \bigcup_{\bb_j\in F(\bb_d)} D_j,\end{gather*} and $[f(\mathcal{C})]\cdot D_j=0$ whenever $\bb_j\notin F(\bb_d)$.
\end{lem}
\begin{proof}
Since $c_1(\alpha)=0$, $f(\mathcal{C})$ should lie in toric divisors of $\bX$. Recall that $\beta_{\mathbf d}$ achieves the minimal Maslov index $2$, and hence there is no disc bubbling.

Suppose $\beta_{\mathbf d}$ is a smooth disk class. Then each sphere component $\mathcal{C}_0$ meeting the disk component $\mathcal{D}$ maps into the divisor $D_{\mathbf d}$ and it should have non-negative intersection with other toric divisors. By \cite[Lemma 4.5]{gi} which easily extends to the simplicial setting, we have the desired statement for $f(\mathcal{C}_0)$.

If $\beta_{\mathbf d}$ is an orbi-disk class, then we can write the corresponding $\bb_d\in N$ as $\bb_d=\sum\limits_{\bb_i\in \sigma}c_i \bb_i$, with $\sum_i c_i=1, c_i\in [0, 1)\cap \mathbb{Q}$. For a sphere component $\mathcal{C}_0$ meeting the disk component $\mathcal{D}$, we have $f(\mathcal{C}_0)\subset \bigcup\limits_{\bb_i\in \sigma} D_i$ and each $\bb_i\in \sigma$ satisfies $\bb_i\in F(\bb_d)$. Hence $f(\mathcal{C}_0)\subset \bigcup\limits_{\bb_i\in F(\bb_d)} D_i$ and $f(\mathcal{C}_0) \cdot D_j =0$ for $\bb_j \notin F(\bb_d)$.

Let $\mathcal{C}_1\subset \mathcal{C}$ be a sphere component meeting $\mathcal{C}_0$, then we have $f(\mathcal{C}_1) \subset F(\bb_j)$ for some $\bb_j \in F(\bb_d)$ by the intersection condition. Now, we can follow the proof of \cite[Proposition~5.6]{cllt} shows that $f(\mathcal{C}_1)\subset \bigcup\limits_{\bb_i\in F(\bb_d)} D_i$. The result follows by repeating this argument for one sphere component at a time.
\end{proof}

Partition $\{\bb_0,\dots , \bb_{m-1}\}\cap F(\bb_d)$ into the disjoint union of two subsets, \begin{gather*}\{\bb_0,\dots , \bb_{m-1}\}\cap F(\bb_d)=F(\bb_d)^c\,\coprod\, F(\bb_d)^{nc},\end{gather*}
where $\bb_i\in F(\bb_d)^c$ if $D_i\subset \bX_{\beta_{\mathbf d}}$ and $\bb_i\in F(\bb_d)^{nc}$ if $D_i\not\subset \bX_{\beta_{\mathbf d}}$.

\begin{lem}\label{lem:in_cy}
Let $f\colon \mathcal{D}\cup \mathcal{C}\to \bX$ be as in Lemma~{\rm \ref{lem:in_div}}. Then we have $f(\mathcal{D}\cup\mathcal{C})\subset \bX_{\beta_{\mathbf d}}$.
\end{lem}
\begin{proof}Certainly $f(\mathcal{D})\subset \bX_{\beta_{\mathbf d}}$. We claim that{\samepage
\begin{gather}\label{eqn:in_cy_div}
 f(\mathcal{C})\subset \bigcup_{\bb_j\in F(\bb_d)^c} D_j,
\end{gather}
from which the lemma follows.}

To see~(\ref{eqn:in_cy_div}), we write $\mathcal{C}=\mathcal{C}_c\cup\mathcal{C}_{nc}$ where $\mathcal{C}_c$ consists of components of $\mathcal{C}$ which lie in $\bigcup\limits_{\bb_j\in F(\bb_d)^c} D_j$, and $\mathcal{C}_{nc}$ consists of the remaining components. Set $A:=f_*[\mathcal{C}_c]$ and $B:=f_*[\mathcal{C}_{nc}]$. Then $\alpha=A+B$. Since $-K_\bX$ is nef and $-K_\bX\cdot \alpha=0$, we have $-K_\bX\cdot A=0=-K_\bX\cdot B$. Write $B=\sum_k c_k B_k$ as an effective linear combination of the classes $B_k$ of irreducible $1$-dimensional torus-invariant orbits in $\bX$. Again because $-K_\bX$ is nef, we have $-K_\bX\cdot B_k=0$ for all $k$. Each~$B_k$ corresponds to an $(n-1)$-dimensional cone $\sigma_k\in \Sigma$. In the expression $B=\sum_k c_k B_k$, there is at least one (non-zero) $B_k$ which is not contained in $\bigcup\limits_{\bb_j\in F(\bb_d)^c} D_j$. As a consequence, either~$\sigma_k$ contains a ray $\mathbb{R}_{\geq 0} \bb_j$ with $\bb_j\notin F(\bb_d)$, or there exists a $\bb_j\notin F(\bb_d)$ such that $\sigma_k$ and $\bb_j$ span an $n$-dimensional cone in~$\Sigma$.

Since $B_k$ is not contained in $\bigcup\limits_{\bb_j\in F(\bb_d)^c} D_j$, we see that if $\bb_i\in F(\bb_d)^c$ then $\bb_i\notin \sigma_k$. Also, $D\cdot B_k\geq 0$ for every toric prime divisor $D\subset \bX$ not corresponding to a ray in $\sigma_k$.

By \cite[Lemma~4.5]{gi} (which easily extends to the simplical setting), we have $D\cdot B_k=0$ for every toric prime divisor $D\subset \bX$ corresponding to an element in $\{\bb_1,\dots,\bb_m\}\setminus F(\sigma_k)$, where $F(\sigma_k)\subset \mathcal{P}$ is the minimal face of $\mathcal{P}$ containing rays in $\sigma_k$. Since the divisors $D\subset \bX$ corresponding to $\{\bb_1,\dots,\bb_m\}\setminus F(\sigma_k)$ span $H^2(\bX)$, we have $B_k=0$, a contradiction.
\end{proof}

Let $\bx = (\mathcal{X}_{\nu_1},\dots, \mathcal{X}_{\nu_l})$ be an $l$-tuple of twisted sectors of $\bX_{\beta_{\mathbf d}}$. Then Lemma~\ref{lem:in_cy} implies that the natural inclusion $\CM^{\rm op, main}_{1,l}(\bX_{\beta_{\mathbf d}}, L,\beta,\bx)\hookrightarrow \CM^{\rm op, main}_{1,l}(\bX, L,\beta,\bx)$ is a bijection. Since $\bX_{\beta_{\mathbf d}}\subset \bX$ is open, the local deformations and obstructions of stable discs in $\bX_{\beta_{\mathbf d}}$ and their inclusion in $\bX$ are isomorphic. It follows that

\begin{prop}\label{prop:op=cl}
The moduli spaces $\CM^{\rm op, main}_{1,l}(\bX, L,\beta,\bx)$ of disks in $\bX$ is isomorphic as Kuranishi spaces to the moduli spaces $\CM^{\rm op, main}_{1,l}(\bX_{\beta_{\mathbf d}}, L,\beta,\bx)$ of disks in $\bX_{\beta_{\mathbf d}}$. Consequently
\begin{gather*}
n_{1,l,\beta}^\bX([{\rm pt}]_{L};\mathbf{1}_{\nu_1},\dots,\mathbf{1}_{\nu_l})= n_{1,l,\beta}^{\bX_{\beta_{\mathbf d}}}([{\rm pt}]_{L};\mathbf{1}_{\nu_1},\dots,\mathbf{1}_{\nu_l}).
\end{gather*}
\end{prop}

Since $\bX_{\beta_{\mathbf d}}$ is a toric Calabi--Yau orbifold, the open Gromov--Witten invariants $n_{1,l,\beta}^{\bX_{\beta_{\mathbf d}}}([{\rm pt}]_{L};\allowbreak \mathbf{1}_{\nu_1},\dots,\mathbf{1}_{\nu_l})$ have been computed in \cite{cclt}. By Proposition~\ref{prop:op=cl}, this gives open Gromov--Witten invariants of~$\bX$. Explicitly they are given as follows.

Using the toric data of ${\bX_{\beta_{\mathbf d}}}$, we define
\begin{gather}\Omega^{\bX_{\beta_{\mathbf d}}}_j := \{d\in \mathbb{K}_{\rm eff} \,|\, \nu(d)=0, \langle D_j,d\rangle \in \Z_{<0}\ \textrm{and} \nonumber\\
\hphantom{\Omega^{\bX_{\beta_{\mathbf d}}}_j := \{}{} \langle D_i,d\rangle \in \Z_{\geq 0}\ \forall\, i\neq j\}, \qquad j = 0, 1, \dots, m-1,\nonumber\\
 \Omega^{\bX_{\beta_{\mathbf d}}}_j:=\{d\in \mathbb{K}_{\rm eff} \,|\, \nu(d)=\bb_j\textrm{ and } \langle D_i,d\rangle \notin \Z_{<0}\ \forall\, i\}, \qquad j = m, m+1, \dots, m'-1,\nonumber\\
A^{\bX_{\beta_{\mathbf d}}}_j(y) := \sum_{d\in \Omega^{\bX_{\beta_{\mathbf d}}}_j}y^d \frac{(-1)^{-\langle D_j,d\rangle-1}(-\langle D_j,d\rangle-1)!}{\prod_{i\neq j}\langle D_i,d\rangle!}, \qquad j = 0, 1, \dots, m-1,\nonumber\\
A^{\bX_{\beta_{\mathbf d}}}_j(y) := \sum_{d\in \Omega^{\bX_{\beta_{\mathbf d}}}_j}y^d \prod_{i=0}^{m'-1}
\frac{\prod_{k=\lceil\langle D_i,d\rangle\rceil}^\infty(\langle D_i,d\rangle-k)}{\prod_{k=0}^\infty(\langle D_i,d\rangle-k)}, \qquad j = m, m+1, \dots, m'-1,\nonumber\\
\label{eqn:toric_mirror_map_X}
\log q_a = \log y_a + \sum_{j=0}^{m-1} Q_{ja}A^{\bX_{\beta_{\mathbf d}}}_j(y), \qquad a = 1,\dots,r',\\
\tau_{\bb_j} = A^{\bX_{\beta_{\mathbf d}}}_j(y), \qquad j=m,\dots,m'-1,\nonumber
\end{gather}

\begin{thm} \label{thm:formula}
If $\beta_{\mathbf{d}}=\beta_{i_0}$ is a basic smooth disk class corresponding to the ray generated by~$\bb_{i_0}$ for some $i_0\in \{0,1,\dots,m-1\}$, then we have
\begin{gather} \sum_{\alpha\in H_2^{\rm eff}(\bX)}\sum_{l\geq 0}\sum_{\nu_1,\dots,\nu_l\in {\rm Box}'(\Sigma_{\beta_{\mathbf d}})^{\operatorname{age}=1}}\frac{\prod\limits_{i=1}^l\tau_{\nu_i}}{l!} n^\bX_{1,l,\beta_{i_0}+\alpha}\left([{\rm pt}]_L; \prod_{i=1}^l \mathbf{1}_{\nu_i}\right)q^\alpha \nonumber\\
\qquad{} = \exp\big({-}A^{\bX_{\beta_{\mathbf d}}}_{i_0}(y(q,\tau))\big)\label{eqn:formula_for_generating_function_sm_disk}
\end{gather}
via the inverse $y = y(q,\tau)$ of the toric mirror map \eqref{eqn:toric_mirror_map_X}.

If $\beta_{\mathbf{d}}=\beta_{\nu_{j_0}}$ is a basic orbi-disk class corresponding to $\nu_{j_0}\in {\rm Box}'(\Sigma)^{\operatorname{age}=1}$ for some $j_0\in \{m,m+1,\dots,m'-1\}$, then we have
\begin{gather}
 \sum_{\alpha\in H_2^{\rm eff}(\bX)}\sum_{l\geq 0}\sum_{\nu_1,\dots,\nu_l\in {\rm Box}'(\Sigma_{\beta_{\mathbf d}})^{\operatorname{age}=1}}\frac{\prod\limits_{i=1}^l\tau_{\nu_i}}{l!}n^\bX_{1,l,\beta_{\nu_{j_0}}+\alpha}([{\rm pt}]_L; \prod_{i=1}^l \mathbf{1}_{\nu_i})q^\alpha \nonumber\\
 \qquad {} = y^{D^\vee_{j_0}} \exp\left(-\sum_{i\notin I_{j_0}}c_{j_0i}A^{\bX_{\beta_{\mathbf d}}}_i(y(q,\tau)) \right),\label{eqn:formula_for_generating_function_orb_disk}
\end{gather}
via the inverse $y=y(q,\tau)$ of the toric mirror map \eqref{eqn:toric_mirror_map_X}, where $D_{j_0}^\vee \in \mathbb{K}_{\rm eff}$ is the class defined in \eqref{eqn:dual_of_D_j}, $I_{j_0}\in \mathcal{A}$ is the anticone of the minimal cone containing $\bb_{j_0} = \nu_{j_0}$ and $c_{j_0i} \in \Q\cap [0,1)$ are rational numbers such that $\bb_{j_0} = \sum\limits_{i\notin I_{j_0}} c_{j_0i}\bb_i$.
\end{thm}
\begin{proof}By Proposition \ref{prop:op=cl}, $n_{1,l,\beta}^\bX([{\rm pt}]_{L};\mathbf{1}_{\nu_1},\dots,\mathbf{1}_{\nu_l})= n_{1,l,\beta}^{\bX_{\beta_{\mathbf d}}}([{\rm pt}]_{L};\mathbf{1}_{\nu_1},\dots,\mathbf{1}_{\nu_l})$, and so the l.h.s.\ of~\eqref{eqn:formula_for_generating_function_sm_disk} is equal to
\begin{gather*} \sum_{\alpha\in H_2^{\rm eff}(\bX_{\beta_{\mathbf d}})}\sum_{l\geq 0}\sum_{\nu_1,\dots,\nu_l\in {\rm Box}'(\Sigma_{\beta_{\mathbf d}})^{\operatorname{age}=1}}\frac{\prod\limits_{i=1}^l\tau_{\nu_i}}{l!}n^{\bX_{\beta_{\mathbf d}}}_{1,l,\beta_{i_0}+\alpha}\left([{\rm pt}]_L; \prod_{i=1}^l \mathbf{1}_{\nu_i}\right)q^\alpha, \end{gather*}
which in turn is equal to $\exp\big({-}A^{\bX_{\beta_{\mathbf d}}}_{i_0}(y(q,\tau))\big)$ by \cite[Theorem~1.4]{cclt}. The deduction for~\eqref{eqn:formula_for_generating_function_orb_disk} is similar.
\end{proof}

To combine all the invariants into a single expression, one defines the orbi-disc potential
\begin{gather}
W=\sum_{\beta_{\mathbf{d}}} \sum_{\alpha\in H_2^{\rm eff}(\bX)}\sum_{l\geq 0}\sum_{\nu_1,\dots,\nu_l\in {\rm Box}'(\Sigma_{\beta_{\mathbf d}})^{\operatorname{age}=1}}\frac{\prod\limits_{i=1}^l\tau_{\nu_i}}{l!} q^\alpha n^\bX_{1,l,\beta_{\mathbf{d}}+\alpha}\left([{\rm pt}]_L; \prod_{i=1}^l \mathbf{1}_{\nu_i}\right) Z^{\beta_{\mathbf{d}}},
\label{eq:W}
\end{gather}
where $\beta_{\mathbf{d}}$ runs over all the basic smooth or orbi-disc classes, and $Z^{\beta_{\mathbf{d}}}$ are monomials associated to $\beta_{\mathbf{d}}$. See \cite[Definition~19]{cllt} for more detail. The above theorem gives an explicit expression of~$W$ via the mirror map.

\begin{Example} $\bP^2/\Z_3$ is a Gorenstein Fano toric orbifold. Its fan and polytope pictures are shown in Fig.~\ref{fig:P2mod3}. It has three toric divisors $D_1$, $D_2$, $D_3$ corresponding to the rays generated by $v_1=(-1,-1)$, $v_2=(2,-1)$, $v_3=(-1,2)$. By pairing with the dual vectors $(1,0)$ and $(0,1)$, the linear equivalence relations are $2D_2-D_3-D_1\sim 0$ and $2D_3-D_2-D_1\sim 0$, and so $D_1\sim D_2\sim D_3$. It has three orbifold points corresponding to the three vertices in the polytope picture. Locally it is $\C^2/\Z_3$ around each orbifold point.
	
\begin{figure}[htb!]\centering
\includegraphics[scale=0.3]{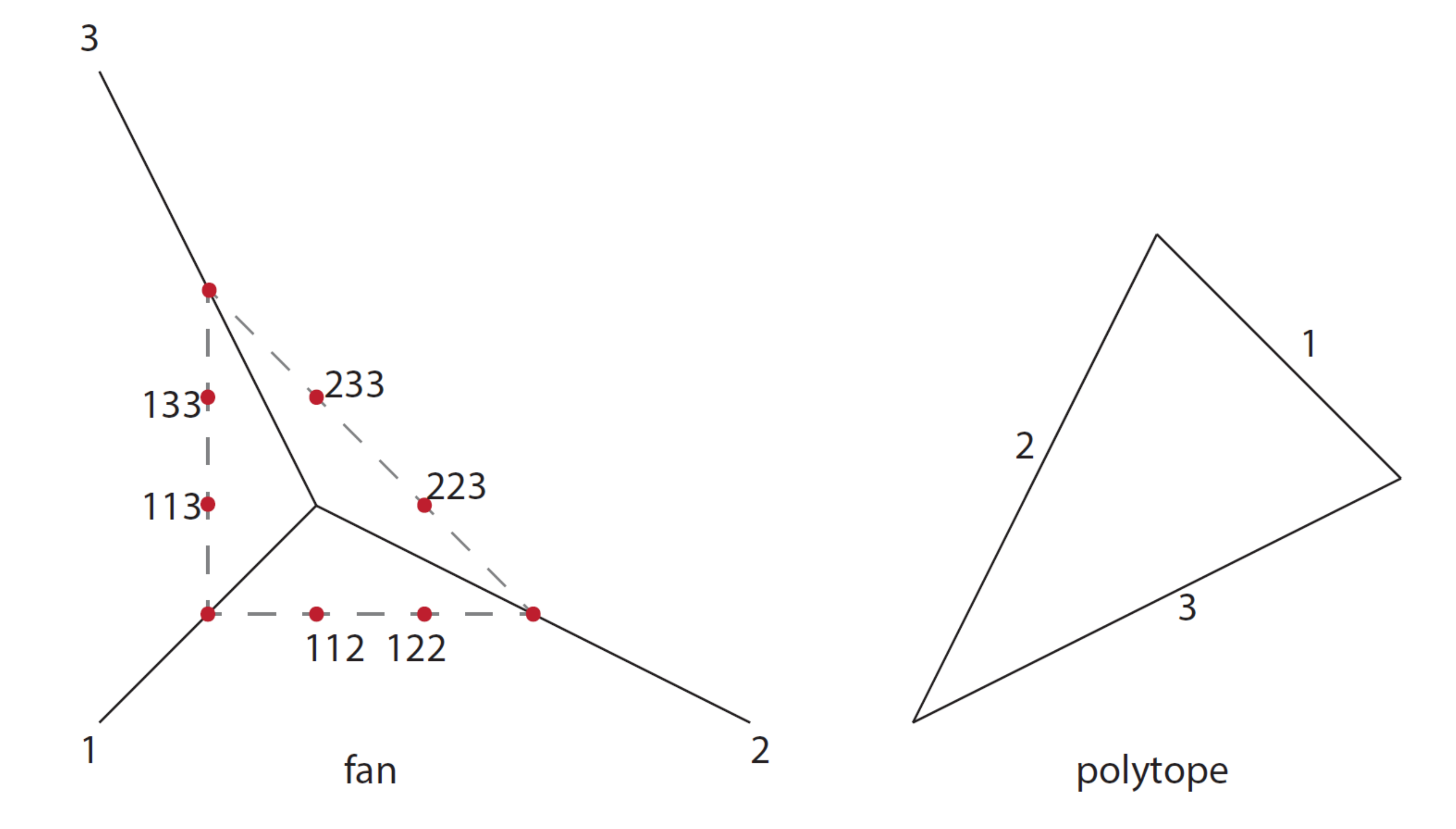}
\caption{The fan and polytope picture for $\bP^2/\Z_3$.}\label{fig:P2mod3}
\end{figure}

Fix a Lagrangian torus fiber. $\bP^2/\Z_3$ has nine basic orbi-disk classes corresponding to the nine lattice points on the boundary of the fan polytope. Three of them are smooth disk classes and denote them by $\beta_1$, $\beta_2$, $\beta_3$. The basic orbi-disk classes corresponding to the two lattice points $(2v_1+v_2)/3$ and $(v_1+2v_2)/3$ are denoted by $\beta_{112}$ and $\beta_{122}$, which pass through the twisted sectors $\nu_{112}$ and $\nu_{122}$ respectively. Then $2 \beta_1 + \beta_2 - 3 \beta_{112}$ (or $2 \beta_2 + \beta_1 - 3 \beta_{122}$) is the class of a constant orbi-sphere passing through the twisted sector $\nu_{112}$ (or $\nu_{122}$ resp.). In particular the area of $ \beta_{112}$ equals to $(2\beta_1+\beta_2)/3$. Other basic orbi-disk classes have similar notations.
	
Theorem~\ref{thm:formula} provides a formula for the open GW invariants $n^\bX_{1,l,\beta_{112}}\Big([{\rm pt}]_L; \prod\limits_{i=1}^l \mathbf{1}_{\nu_i}\Big)$ where $\nu_i$ is either $\nu_{112}$ or $\nu_{122}$ for each~$i$. To write down the invariants more systematically, we consider the open GW potential as follows.
	
Let $q$ be the K\"ahler parameter of the smooth sphere class $\beta_1+\beta_2+\beta_3 \in H_2\big(\bP^2/\Z_3\big)$. The basic orbi-disk classes correspond to monomials in the disk potential $q^{\beta} z^{\partial \beta}$, where $q^{\beta_1}=q^{\beta_2}=q^{\beta_{112}}=q^{\beta_{122}}=1$, $q^{\beta_3}=q^{\beta_1+\beta_2+\beta_3}=q$, $q^{\beta_{223}}=q^{(2\beta_2+\beta_3)/3}=q^{\beta_3/3}=q^{1/3}$, and similar for other basic orbi-disk classes. The K\"ahler parameters corresponding to the twisted sectors $\nu_{112}$, $\nu_{122}$ are denoted as $\tau_{112},\tau_{122}$ (and similar for other twisted sectors).
	
By \cite[Example~1, Section~6.5]{cclt}, the open GW potential for $\C^2/\Z_3$ is given by
\[
w(z-\kappa_0(\tau_{112},\tau_{122}))(z-\kappa_1(\tau_{112},\tau_{122}))(z-\kappa_2(\tau_{112},\tau_{122}))
\] where
\begin{gather*}
\kappa_k(\tau_1,\tau_2)=\zeta^{2k+1}\prod_{r=1}^{2}\exp\left(\frac{1}{3}\zeta^{(2k+1)r}\tau_r \right), \qquad \zeta:=\exp\big(\pi\sqrt{-1}/3\big).
\end{gather*}
By Proposition \ref{prop:op=cl}, the disk invariants of $\bP^2/\Z_3$ equal to those of $\C^2/\Z_3$. Thus the open GW potential of $\bP^2/\Z_3$ is given by
\begin{gather*}
W = z^{-1}w^{-1}(z-\kappa_0(\tau_{112},\tau_{122}))(z-\kappa_1(\tau_{112},\tau_{122}))(z-\kappa_2(\tau_{112},\tau_{122})) \\
\hphantom{W =}{} + z^{-1}w^{-1}\big(q^{1/3}w-\kappa_0(\tau_{113},\tau_{133})\big) \big(q^{1/3}w-\kappa_1(\tau_{113},\tau_{133})\big)\big(q^{1/3}w-\kappa_2(\tau_{113},\tau_{133})\big)\\
\hphantom{W =}{} + z^{2}w^{-1}\big(q^{1/3}z^{-1}w-\kappa_0(\tau_{223},\tau_{233})\big) \big(q^{1/3}z^{-1}w-\kappa_1(\tau_{223},\tau_{233})\big)\\
\hphantom{W =}{}\times \big(q^{1/3}z^{-1}w-\kappa_2(\tau_{223},\tau_{233})\big)
- z^{-1}w^{-1} - z^{2}w^{-1} - q z^{-1}w^{2}.
\end{gather*}
Then the generating functions of open orbifold GW for $\beta_{112}$ and $\beta_{122}$ are given by the coefficients of $w^{-1}$ and $zw^{-1}$ in $W$ respectively. The first few terms are given by the following table.

\begin{table}[h!]\centering\small
	\begin{tabular}{|c|c|c|c|c|c|c|c|}
		\hline
		$n_{(a,b)}$ & $a=0$ & $a=1$ & $a=2$ & $a=3$ & $a=4$ & $a=5$ & $a=6$ \\
		\hline
		$b=0$ & $0$ & $1$ & $0$ & $0$ & $1/648$ & $0$ & $0$ \\
		\hline
		$b=1$ & $0$ & $0$ & $-1/18$ & $0$ & $0$ & $-1/29160$ & $0$ \\
		\hline
		$b=2$ & $1/6$ & $0$ & $0$ & $1/972$ & $0$ & $0$ & $1/3149280$ \\
		\hline
		$b=3$ & $0$ & $-1/162$ & $0$ & $0$ & $-1/104976$ & $0$ & $0$ \\
		\hline
		$b=4$ & $0$ & $0$ & $1/11664$ & $0$ & $0$ & $1/18895680$ & $0$ \\
		\hline
		$b=5$ & $\!-1/9720\!$ & $0$ & $0$ & $-1/1574640$ & $0$ & $0$ & $\!-1/5101833600\!$ \\
		\hline
		$b=6$ & $0$ & $1/524880$ & $0$ & $0$ & $1/340122240$ & $0$ & $0$ \\
		\hline		
	\end{tabular}
\end{table}

In the above table,
\begin{gather*} n_{(a,b)} = n_{1,a+b,\beta_{112}}\big([{\rm pt}]_L;\mathbf{1}_{\nu_{112}}^{\otimes a},\mathbf{1}_{\nu_{122}}^{\otimes b}\big) = n_{1,a+b,\beta_{122}}\big([{\rm pt}]_L;\mathbf{1}_{\nu_{112}}^{\otimes b},\mathbf{1}_{\nu_{122}}^{\otimes a}\big).\end{gather*}
We observe that all invariants satisfy `reciprocal integrality', namely their reciprocals are integers. Moreover, all these integers are divisible by~$6$. $n_{(k,k)}=0$. Furthermore, the sign is alternating with respect to~$b$.
\end{Example}

\subsection*{Acknowledgments}
K.~Chan was supported by a Hong Kong RGC grant CUHK14314516 and direct grants from CUHK. C.-H.~Cho was supported by the NRF grant funded by the Korea government(MSIT) (No.~2017R1A22B4009488). S.-C.~Lau was partially supported by the Simons collaboration grant \#580648. N.C.~Leung was supported by Hong Kong RGC grants CUHK14302215 \& CUHK14303516 and direct grants from CUHK. H.-H.~Tseng was supported in part by NSF grant DMS-1506551.

\pdfbookmark[1]{References}{ref}
\LastPageEnding


\begin{thebibliography}{99}
\footnotesize\itemsep=0pt
\bibitem{auroux07}
Auroux D., Mirror symmetry and {$T$}-duality in the complement of an
 anticanonical divisor, \textit{J.~G\"{o}kova Geom. Topol.} \textbf{1} (2007),
 51--91, \href{https://arxiv.org/abs/0706.3207}{arXiv:0706.3207}.

\bibitem{BCS}
Borisov L.A., Chen L., Smith G.G., The orbifold {C}how ring of toric
 {D}eligne--{M}umford stacks, \href{https://doi.org/10.1090/S0894-0347-04-00471-0}{\textit{J.~Amer. Math. Soc.}} \textbf{18} (2005),
 193--215, \href{https://arxiv.org/abs/math.AG/0309229}{arXiv:math.AG/0309229}.

\bibitem{cclt-OCRC}
Chan K., Cho C.-H., Lau S.-C., Tseng H.-H., Lagrangian {F}loer superpotentials and
 crepant resolutions for toric orbifolds, \href{https://doi.org/10.1007/s00220-014-1948-6}{\textit{Comm. Math. Phys.}}
 \textbf{328} (2014), 83--130, \href{https://arxiv.org/abs/1208.5282}{arXiv:1208.5282}.

\bibitem{cclt}
Chan K., Cho C.-H., Lau S.-C., Tseng H.-H., Gross fibrations, {SYZ} mirror
 symmetry, and open {G}romov--{W}itten invariants for toric {C}alabi--{Y}au
 orbifolds, \href{https://doi.org/10.4310/jdg/1463404118}{\textit{J.~Differential Geom.}} \textbf{103} (2016), 207--288,
 \href{https://arxiv.org/abs/1306.0437}{arXiv:1306.0437}.

\bibitem{cllt}
Chan K., Lau S.-C., Leung N.C., Tseng H.-H., Open {G}romov--{W}itten invariants,
 mirror maps, and {S}eidel representations for toric manifolds, \href{https://doi.org/10.1215/00127094-0000003X}{\textit{Duke
 Math.~J.}} \textbf{166} (2017), 1405--1462, \href{https://arxiv.org/abs/1209.6119}{arXiv:1209.6119}.

\bibitem{CR01}
Chen W., Ruan Y., Orbifold {G}romov--{W}itten theory, in Orbifolds in
 Mathematics and Physics ({M}adison, {WI}, 2001), \href{https://doi.org/10.1090/conm/310/05398}{\textit{Contemp. Math.}},
 Vol.~310, Amer. Math. Soc., Providence, RI, 2002, 25--85,
 \href{https://arxiv.org/abs/math.AG/0103156}{arXiv:math.AG/0103156}.

\bibitem{CR04}
Chen W., Ruan Y., A new cohomology theory of orbifold, \href{https://doi.org/10.1007/s00220-004-1089-4}{\textit{Comm. Math.
 Phys.}} \textbf{248} (2004), 1--31, \href{https://arxiv.org/abs/math.AG/0004129}{arXiv:math.AG/0004129}.

\bibitem{cho06}
Cho C.-H., Oh Y.-G., Floer cohomology and disc instantons of {L}agrangian torus
 fibers in {F}ano toric manifolds, \href{https://doi.org/10.4310/AJM.2006.v10.n4.a10}{\textit{Asian~J. Math.}} \textbf{10} (2006),
 773--814, \href{https://arxiv.org/abs/math.SG/0308225}{arXiv:math.SG/0308225}.

\bibitem{CP}
Cho C.-H., Poddar M., Holomorphic orbi-discs and {L}agrangian {F}loer cohomology
 of symplectic toric orbifolds, \href{https://doi.org/10.4310/jdg/1406137695}{\textit{J.~Differential Geom.}} \textbf{98} (2014), 21--116, \href{https://arxiv.org/abs/1206.3994}{arXiv:1206.3994}.

\bibitem{CS}
Cho C.-H., Shin H.-S., Chern--{W}eil {M}aslov index and its orbifold analogue,
 \href{https://doi.org/10.4310/AJM.2016.v20.n1.a1}{\textit{Asian~J. Math.}} \textbf{20} (2016), 1--19, \href{https://arxiv.org/abs/1202.0556}{arXiv:1202.0556}.

\bibitem{CCIT_toricDM}
Coates T., Corti A., Iritani H., Tseng H.-H., A mirror theorem for toric stacks,
 \href{https://doi.org/10.1112/S0010437X15007356}{\textit{Compos. Math.}} \textbf{151} (2015), 1878--1912, \href{https://arxiv.org/abs/1310.4163}{arXiv:1310.4163}.

\bibitem{CLS_toricbook}
Cox D.A., Little J.B., Schenck H.K., Toric varieties, \textit{Graduate Studies
 in Mathematics}, Vol.~124, \href{https://doi.org/10.1090/gsm/124}{Amer. Math. Soc.}, Providence, RI, 2011.

\bibitem{FOOO_book}
Fukaya K., Oh Y.-G., Ohta H., Ono K., Lagrangian intersection {F}loer theory:
 anomaly and obstruction, \textit{AMS/IP Studies in Advanced Mathematics},
 Vol.~46, Amer. Math. Soc., Providence, RI, 2009.

\bibitem{FOOO1}
Fukaya K., Oh Y.-G., Ohta H., Ono K., Lagrangian {F}loer theory on compact toric
 manifolds.~{I}, \href{https://doi.org/10.1215/00127094-2009-062}{\textit{Duke Math.~J.}} \textbf{151} (2010), 23--174,
 \href{https://arxiv.org/abs/0802.1703}{arXiv:0802.1703}.

\bibitem{FOOO2}
Fukaya K., Oh Y.-G., Ohta H., Ono K., Lagrangian {F}loer theory on compact toric
 manifolds~{II}: bulk deformations, \href{https://doi.org/10.1007/s00029-011-0057-z}{\textit{Selecta Math. (N.S.)}} \textbf{17}
 (2011), 609--711, \href{https://arxiv.org/abs/0810.5654}{arXiv:0810.5654}.

\bibitem{FOOO12}
Fukaya K., Oh Y.-G., Ohta H., Ono K., Technical details on Kuranishi structure
 and virtual fundamental chain, \href{https://arxiv.org/abs/1209.4410}{arXiv:1209.4410}.

\bibitem{FOOO10b}
Fukaya K., Oh Y.-G., Ohta H., Ono K., Lagrangian {F}loer theory and mirror
 symmetry on compact toric manifolds, \textit{Ast\'{e}risque} \textbf{376}
 (2016), vi+340~pages, \href{https://arxiv.org/abs/1009.1648}{arXiv:1009.1648}.

\bibitem{gi}
Gonz\'alez E., Iritani H., Seidel elements and mirror transformations,
 \href{https://doi.org/10.1007/s00029-011-0080-0}{\textit{Selecta Math. (N.S.)}} \textbf{18} (2012), 557--590, \href{https://arxiv.org/abs/1103.4171}{arXiv:1103.4171}.

\bibitem{GS07}
Gross M., Siebert B., From real affine geometry to complex geometry,
 \href{https://doi.org/10.4007/annals.2011.174.3.1}{\textit{Ann. of Math.}} \textbf{174} (2011), 1301--1428,
 \href{https://arxiv.org/abs/math.AG/0703822}{arXiv:math.AG/0703822}.

\bibitem{Hausel-Sturmfels}
Hausel T., Sturmfels B., Toric hyper{K}\"ahler varieties, \textit{Doc. Math.}
 \textbf{7} (2002), 495--534, \href{https://arxiv.org/abs/math.AG/0203096}{arXiv:math.AG/0203096}.

\bibitem{iritani09}
Iritani H., An integral structure in quantum cohomology and mirror symmetry for
 toric orbifolds, \href{https://doi.org/10.1016/j.aim.2009.05.016}{\textit{Adv. Math.}} \textbf{222} (2009), 1016--1079,
 \href{https://arxiv.org/abs/0903.1463}{arXiv:0903.1463}.

\bibitem{Jiang08}
Jiang Y., The orbifold cohomology ring of simplicial toric stack bundles,
 \href{https://doi.org/10.1215/ijm/1248355346}{\textit{Illinois~J. Math.}} \textbf{52} (2008), 493--514,
 \href{https://arxiv.org/abs/math.AG/0504563}{arXiv:math.AG/0504563}.

\bibitem{McDuff-Wehrheim12}
McDuff D., Wehrheim K., Smooth {K}uranishi atlases with isotropy, \href{https://doi.org/10.2140/gt.2017.21.2725}{\textit{Geom.
 Topol.}} \textbf{21} (2017), 2725--2809, \href{https://arxiv.org/abs/1508.01556}{arXiv:1508.01556}.

\bibitem{RS}
Ruddat H., Siebert B., Period integrals from wall structures via tropical
 cycles, canonical coordinates in mirror symmetry and analyticity of toric
 degenerations, \href{https://doi.org/10.1007/s10240-020-00116-y}{\textit{Publ. Math. Inst. Hautes \'Etudes Sci.}}, {t}o appear,
 \href{https://arxiv.org/abs/1907.03794}{arXiv:1907.03794}.

\bibitem{SYZ}
Strominger A., Yau S.-T., Zaslow E., Mirror symmetry is {$T$}-duality,
 \href{https://doi.org/10.1016/0550-3213(96)00434-8}{\textit{Nuclear Phys.~B}} \textbf{479} (1996), 243--259, \href{https://arxiv.org/abs/hep-th/9606040}{arXiv:hep-th/9606040}.

\end{thebibliography}
\end{document}